\documentclass[11pt,reqno]{article}
\usepackage[utf8]{inputenc}
\usepackage{bm, amsmath, mathrsfs, amssymb, amsthm, amscd, amsfonts, graphicx, color,epsfig,latexsym,eucal,layout,fancyhdr, colonequals, enumerate, hyperref, fontenc, hyperref, graphics, a4wide, verbatim, psfrag}

\newcounter{EQNR}
\renewcommand{\contentsname}{Contents\\
{\footnotesize\normalfont(A table of contents should normally not be included)}}

\newtheorem{theorem}{Theorem}
\newtheorem{corollary}[theorem]{Corollary}

\newtheorem{example}[theorem]{Example}
\newtheorem{lemma}[theorem]{Lemma}
\newtheorem{proposition}[theorem]{Proposition}

\newtheorem{remark}[theorem]{Remark}

\begin{document}

\title{The resolvent kernel on the discrete circle and twisted cosecant sums}
\author{Jay Jorgenson\footnote{Supported by
PSC-CUNY Awards, which are jointly funded by the Professional Staff Congress and The City University of New York.} \and Anders
 Karlsson\footnote{Supported in part by the Swedish Research Council grant 104651320
 and the Swiss NSF grants 200020-200400 and 200021-212864.} \and
Lejla Smajlovi\'{c}}
\maketitle

\begin{abstract}\noindent
Let $X_m$ denote the discrete circle with $m$ vertices.  For $x,y\in X_{m}$ and complex $s$,
let $G_{X_m,\chi_{\beta}}(x,y;s)$  be the resolvent kernel associated
to the combinatorial Laplacian
which acts on the space of functions on $X_{m}$ that are twisted by a character $\chi_{\beta}$.  We will
compute $G_{X_m,\chi_{\beta}}(x,y;s)$ in two different ways.  First, using the spectral expansion of the Laplacian, we show that
$G_{X_m,\chi_{\beta}}(x,y;s)$ is a generating function for certain trigonometric sums involving powers of
the cosecant function; by choosing $\beta$
or $s$ appropriately, the sums in question involve powers of the secant function.
Second, by viewing $X_{m}$ as a quotient space of $\mathbb{Z}$, we prove
that $G_{X_m,\chi_{\beta}}(x,y;s)$ is a rational function which is given in terms of Chebyshev polynomials.
From the existence and uniqueness of $G_{X_m,\chi_{\beta}}(x,y;s)$, these two evaluations are equal.
From the resulting identity, we obtain a means by which one can obtain
explicit evaluations of cosecant and secant sums.
The identities we prove depend on a number of parameters, and when we specialize the values of these parameters we obtain
several previously known formulas.  Going further, we derive a recursion formula for special values of the $L$-functions
associated to the cycle graph $X_{m}$, thus answering a question from \cite{XZZ22}.
\end{abstract}

\section{Introduction}

Finite trigonometric sums of the type
$$
C_{m}(n):=\sum_{j=1}^{m-1} \frac{1}{\sin^{n}\left(j\pi /m\right)}
$$
have a long history and appear in various contexts. Two early points of reference are in Eisenstein's work and in
the study of Dedekind sums \cite{BY02}.
Modern appearances of these sums include the Hirzebruch signature defects and the Verlinde formulas in topology
and mathematical physics \cite{HZ74, Ve88, Do92, Za96}, resistance in networks \cite{Wu04, EW09, Ch12, Ch14b}
as well as modeling angles in proteins and circular genomes \cite{F-DG-D14}. Many further instances are described
in \cite{BY02}, such as the chiral Potts model in
statistical physics \cite{MO96}, \cite{Ch14a}. Finite trigonometric sums are also related to Dedekind and
Hardy sums and their generalizations.  Several of those sums seem not to have known evaluations, but it is
possible to establish reciprocal relations, see for example \cite{BC13}, \cite{Ch18} or \cite{MS20}.

The sums $C_{m}(n)$ are also discrete analogs of the Riemann zeta function, as observed
by Dowker in \cite{Do92} and further developed in \cite{FK17}. This link is already implicitly present in \cite{Ap73} where the
asymptotics of the cotangent sums
$$
\sum_{k=1}^m \cot^{n}(k\pi /(2m+1))
$$
as $m\rightarrow\infty$ are used to evaluate the Riemann zeta values $\zeta(n)$, ultimately recovering Euler's formula in case $n$ is even.
These computations indicate the delicate nature of these trigonometric sums, including $C_m(n)$, because the values of
$\zeta(2n)$ are known while the values of $\zeta(2n+1)$ are far from understood.

Recent contributions to the evaluation of trigonometric sums include \cite{AH18}, \cite{AZ22}, \cite{GLY22} \cite{CHJSV23} as well as \cite{XZZ22}.
In \cite{XZZ22}, the authors found a precise formula for $\zeta(2n)$ as a finite linear combination, with universal constants,
of the sums $C_m(2k)$ for $0<k<n$.  These formulas do not involve asymptotic expansions. In the same paper, the authors obtain similar formulas for
special values of Dirichlet $L$-functions and ask for a direct evaluation of the corresponding twisted trigonometric sums.
One of the results in the present article is to provide an answer to this question posed by \cite{XZZ22}; see section \ref{sec: mult twist}.

We use the notation of the cosecant function $\csc(x)=1/\sin(x)$ and the secant function $\sec(x) = 1/\cos(x)$ throughout this article.
Let $m$ and $n$ be positive integers, and $\beta$ be a positive real number, which we call a shift.  Define
\begin{equation}\label{eq: cosec sum basic}
 C_{m}(\beta,n):=\sum_{j=\delta(\beta)}^{m-1} \csc^{n}\left(\frac{(j+\beta)}{m}\pi\right),
\end{equation}
where $\delta(\beta)=1$ if $\beta\in\mathbb{Z}$ and $0$, otherwise.  Similarly, we define the alternating sums
\begin{equation}\label{eq: cosec sum alternating}
 C_{m}^{\mathrm{alt}}(\beta,n):=\sum_{j=\delta(\beta)}^{m-1} (-1)^j\csc^{n}\left(\frac{(j+\beta)}{m}\pi\right).
\end{equation}
Chu and Marini proved in \cite{CM99}, among other formulas, that for any positive integer $m$ one has that
\begin{equation}\label{eq:generating_series_one}
\sum_{n=1}^{\infty}C_m(0,2n)y^{2n}=1-\frac{my}{\sqrt{1-y^2}}\cot(m\arcsin(y))
\end{equation}
while for even positive integers $m$ we have that
\begin{equation}\label{eq:generating_series_two}
\sum_{n=1}^{\infty} C_{m}^{\mathrm{alt}}(0,2n)y^{2n}=1-\frac{my}{\sqrt{1-y^2}}\csc(m\arcsin(y)).
\end{equation}
In words, the sequence of series \eqref{eq: cosec sum basic} and \eqref{eq: cosec sum alternating} can be
used to form a generating functions \eqref{eq:generating_series_one} and \eqref{eq:generating_series_two} which can be
explicitly computed.  As such, one can evaluate any given
series \eqref{eq: cosec sum basic} or \eqref{eq: cosec sum alternating} by computing the corresponding
coefficient in the Taylor expansion on the right-hand-side of \eqref{eq:generating_series_one} or
\eqref{eq:generating_series_two}, respectively.

Formulas for other generating functions have been derived in \cite{CM99}.  For example, in \cite{CM99}
the authors evaluate the series defined by using
the sequence of terms formed from the series in \eqref{eq: cosec sum basic} and \eqref{eq: cosec sum alternating}
when $\beta=1/2$.   More generally, it was proved in \cite{WZ07} that
\begin{equation}\label{eq:generating_series_three}
\sum_{n=1}^{\infty}C_m(\beta,2n)y^{2n}= \frac{my}{\sqrt{1-y^2}}\frac{\sin(2m\arcsin(y))}{\cos(2m\arcsin(y))-\cos2m\beta}.
\end{equation}
A similar formula also is deduced for the generating function of the alternating cosecant sums \eqref{eq: cosec sum alternating}.
Other authors have studied twists of powers of cosecants by cosine function, see for example \cite{Do92}, Section 3 of \cite{BY02}
as well as Section 3 of \cite{He20}.  Those authors also study secant sums and derive similar results.

In this paper we will study the cosecant sums with shift $\beta\geq 0$ and twisted by an additive character.
In doing so, we also derive results for analogously defined secant sums by suitably adjusting the shift $\beta$.

More precisely, let $m>1$ be an integer, 
let $\beta$ be a positive nonintegral real number 
and take $r\in \{-(m-1),\ldots,0,\ldots, (m-1)\}$. 
We define (the average value of) the twisted cosecant sums associated to those parameters and a positive integer $n$ by

\begin{equation}\label{eq: cosec sum at n - general}
C_{m,r}(\beta,n):=\frac{1}{m}\sum_{j=0}^{m-1} \csc^{2n}\left(\frac{j+\beta}{m}\pi\right) e^{2\pi i r j/ m}.
\end{equation}
The cosecant sums without the shift $\beta$ are defined as
\begin{equation}\label{eq: non-twist cosec sum at n - general}
C_{m,r}(n):=\frac{1}{m}\sum_{j=1}^{m-1} \csc^{2n}\left(\frac{j}{m}\pi\right) e^{2\pi i r j/m}.
\end{equation}
Note that to get \eqref{eq: non-twist cosec sum at n - general} from \eqref{eq: cosec sum at n - general},
one omits the term where $j=0$ and then sets $\beta = 0$.
The sums \eqref{eq: non-twist cosec sum at n - general} appear in the formulas deduced in \cite{Ta92}  for the dimensions
of a certain complex vector space at level $k$ associated to a labeled Riemann surface of genus $g\geq 2$. Specifically, in
statement (12) of \cite{Ta92} the aforementioned dimension is expressed in terms of \eqref{eq: non-twist cosec sum at n - general} with
$r=0$, $m=k+2$ and $n=g-1$, while in statement (18) of \cite{Ta92}, the appropriate dimension of the "twisted" space is expressed in
terms of \eqref{eq: non-twist cosec sum at n - general} with even $k$,  $m=k+2$, $r=m/2$, and $n=g-1$. Both expressions
are special cases of Verlinde sums; see for example \cite[pp. 11, 14]{Ve01}. 

Let $m$ and $r$ be as above, and let $\alpha$ be such that $\alpha - \frac{m}{2} \notin \mathbb{Z}$. The (average value of) the twisted secant sums associated to those parameters and a positive integer $n$ are defined as
\begin{equation}\label{eq: sec sum at n - general}
S_{m,r}(\alpha,n):=\frac{1}{m}\sum_{j=0}^{m-1} \sec^{2n}\left(\frac{j+\alpha}{m}\pi\right) e^{2\pi i rj /m}.
\end{equation}
The (average) secant sums without the shift $\alpha$ are defined as
\begin{equation}\label{eq: non-twist sec sum at n - general}
S_{m,r}(n):=\frac{1}{m}\sum_{j\in\{0,\ldots, m-1\}\setminus\{j_m\}^\ast} \sec^{2n}\left(\frac{j }{m}\pi\right) e^{2\pi i r j/m},
\end{equation}
where $\{j_m\}^\ast$ is the empty set if $m$ is odd and contains the single number $j_m$ such that $j_m \equiv \frac{m}{2}\, (\mathrm{mod\, } m)$
in the case when  $m$ is even.

In this article we will also study powers (that are not necessarily even) of cosecant and secant functions evaluated at doubled arguments.
We will derive an explicit evaluation of their generating functions as well as a finite recursion formula for computation.
More precisely,  for real number $\alpha$ such that $\alpha \notin \mathbb{Z}$ when $m\equiv 0 \,(\mathrm{mod}\, 4)$,
$\alpha \notin \mathbb{Z}+\frac{1}{2}$ when $m\equiv 2 \,(\mathrm{mod}\, 4)$ and  $2\alpha  \notin \mathbb{Z}+\frac{1}{2}$ when
$m$ is odd, we will study the sum
\begin{equation}\label{eq: sec sum at even arg n - general}
\tilde{S}_{m,r}(\alpha,n):=\frac{1}{m}\sum_{j=0}^{m-1} \sec^{n}\left(\frac{2(j+\alpha)}{m}\pi\right) e^{2\pi i r j /m}.
\end{equation}
When $m$ is not divisible by $4$ and by taking $\alpha=0$ in \eqref{eq: sec sum at even arg n - general}, we immediately obtain
the secant sums of double argument without the shift. If  $m\equiv 0 \,(\mathrm{mod}\, 4)$, then one needs to exclude the
value of $j$ for which  $j\equiv \frac{m}{4}\, (\mathrm{mod\, } m)$ from the range $0,\ldots,m-1$ of summation in
\eqref{eq: sec sum at even arg n - general}. Such a sum equals zero when $n+r$ is odd and equals $S_{m/2,r/2}(n/2)$
when both  $r$ and $n$ are even. We leave the study of the special case $\alpha=0$ of the sum $\tilde{S}_{m,r}(\alpha,n)$
when $m\equiv 0 \,(\mathrm{mod}\, 4)$ and both $r,\,n$ are odd to the interested reader.

For any real number $\beta$ such that $2\beta \notin\mathbb{Z}$ when $m$ is odd and such that $\beta \notin \mathbb{Z}$
when $m$ is even, we consider the  sums
\begin{equation}\label{eq: cosec sum at even arg n - general}
\tilde{C}_{m,r}(\beta,n):=\frac{1}{m}\sum_{j=0}^{m-1} \csc^{n}\left(\frac{2(j+\beta)}{m}\pi\right) e^{2\pi i r j /m}.
\end{equation}

The (average) cosecant sums of double argument, without the shift $\beta$ are defined, with the definition of $\{j_m\}^\ast$  as above by
\begin{equation}\label{eq: non-twist cosec sum double at n - general}
\tilde{C}_{m,r}(n):=\frac{1}{m}\sum_{j\in\{1,\ldots, m-1\}\setminus\{j_m\}^\ast} \csc^{n}\left(\frac{2j }{m}\pi\right) e^{2\pi i r j /m}.
\end{equation}

The above defined cosecant and secant sums, both with and without the shift $\beta$ or twist by an additive character, have been extensively
studied using various methods.  For example, the authors in \cite{CM99}, \cite{BY02}, \cite{WZ07}, \cite{CS12}, \cite{Do15} used contour integration,
generating series and partial fraction decomposition to evaluate those sums as well as their generating functions.  The approach in \cite{dFGK17},
\cite{dFGK18} uses recurrence relations and generating series, while \cite{He20} starts with Taylor series expansions of powers of tangent and
cotangent. In \cite{AH18} the starting point is to use varioius results in the theory of certain special functions.  Also,
a discrete form of sampling theorem was used in \cite{Ha08},  while \cite{AZ22} describes an ``automated approach'' for
proving some trigonometric identities.

In this article, we offer a different point of view and also study a more general situation, which includes
series which may include a twist by an additive character. The approach is inspired by Dowker's computation of the heat
kernel on a generalized cone \cite{Do89} and the key observation is that the resolvent for the twisted heat kernel
on a cycle graph can be viewed as a generating function for certain secant and cosecant sums.

Let us now describe our approach and state our main results.

\subsection{Overview of methods and illustration of results}

Let $X_m$ denote the weighted Cayley graph with vertex set $\mathbb{Z}/m\mathbb{Z}$, generator set $S=\{-1,1\}$, and weights given by the uniform probability distribution on $S$. Let $\beta\in\mathbb{R}$ be an arbitrary real parameter. Our starting point is the ``twisted by an additive character'' $\chi_{\beta}(x):= \exp(2\pi i \beta x)$ heat kernel on $X_m$.  We compute the heat kernel using two different means. First, we employ the method of averaging, by
which we mean that we view $\mathbb{Z}/m\mathbb{Z}$ as being covered by $\mathbb{Z}$ and then we sum the heat kernel on $\mathbb{Z}$
by the covering group $m\mathbb{Z}$.  Second, we use the discrete spectral expansion of the standard Laplacian on $X_m$.
Since the heat kernel under consideration is unique, the two different evaluations yield an identity.  From this identity, we
then compute the resolvent kernel  $G_{X_m,\chi_{\beta}}$ twisted by the character $\chi_{\beta}$  (or twisted
Green's function, see \cite{CY00}) for the Laplace operator on the graph $X_m$.  Essentially, the resolvent kernel is equal
to the Laplace transform in the time variable of the heat kernel.

The above calculations yield
an explicit identity for the resolvent kernel  $G_{X_m,\chi_{\beta}}(x,y;s)$ for real $\beta$ which is
obtained by equating the two evaluations.  The resulting formula
admits a meromorphic continuation to all complex values of $s$.  We
then determine its analytic properties for different values of real parameter $\beta$ at $s=0$ and $s=-1$. The properties at
$s=0$ will yield results related to twisted even powers of secants and cosecants.  The properties at $s=-1$ will yield results
related to twisted, though not necessarily even, powers of shifted secants and cosecants at double arguments. Going
further, we will apply
the Gauss formula for primitive Dirichlet characters to get an explicit evaluation of the
Dirichlet $L$-function associated to the cycle graph at positive integers.

\subsubsection{Generating functions for twisted sums of even powers}

To illustrate our results let us state the first main theorem.  With the notation as above,
let $\ell\in\{0,\ldots,m-1\}$ be such that $\ell\equiv\, r\, (\mathrm{mod}\, m)$.
For $\beta\notin \mathbb{Z}$ define the generating functions
\begin{equation*}
f_{m,r}(s,\beta)= \sum_{n=0}^{\infty} C_{m,r}(\beta,n+1)s^n
\end{equation*}
and
$$
f_{m,r}(s)= \sum_{n=0}^{\infty} C_{m,r}(n+1)s^n
$$
for the cosecant sums \eqref{eq: cosec sum at n - general} and \eqref{eq: non-twist cosec sum at n - general}.
The first main result is the following theorem.

\begin{theorem}\label{thm: main}
For all complex $s$ with $\vert s \vert$  sufficiently small, the series which defines $f_{m,r}(s,\beta)$
converges uniformly and absolutely.  Furthermore, the function $f_{m,r}(s,\beta)$ admits a meromorphic
continuation to all complex $s$, and we have that
\begin{equation*}
f_{m,r}(s,\beta)= 2e^{-2\pi i \beta \ell /m}\cdot \frac{U_{m-\ell-1}(1-2s)+e^{2\pi i \beta }U_{\ell-1}(1-2s)}{T_{m}(1-2s)-\cos2\pi \beta },
\end{equation*}
where $T_n$ and $U_n$ denote the Chebyshev polynomials of the first and the second kind, with the
convention that $U_{-1}(x) \equiv 0$.

Similarly, for all complex $s$ with $\vert s \vert$  sufficiently small, the series which defines $f_{m,r}(s)$
converges uniformly and absolutely.  Furthermore, the function $f_{m,r}(s)$ admits a meromorphic
continuation to all complex $s$, and we have that
\begin{equation} \label{eq: gen funct beta =0}
f_{m,r}(s) = 2\frac{U_{m-\ell-1}(1-2s)+U_{\ell-1}(1-2s)}{T_{m}(1-2s)-1}+\frac{1}{ms}.
\end{equation}
\end{theorem}
For relevant information about Chebyshev polynomials see for example \cite[Section 8.94]{GR07}.  For the
convenience of the reader, we state the most relevant results regarding Chebyshev polynomials
in the concluding section \ref{appendix: Cheb pol}.
With the contents of section \ref{appendix: Cheb pol} to the side, we can give a simple qualitative description
of Theorem \ref{thm: main}, which is the following:

\noindent
\it Both of the power series $f_{m,r}(s,\beta)$ and $f_{m,r}(s)$ are, in fact,
rational functions is $s$ with numerators and denominators given in terms of classical Chebyshev polynomials
which are precisely defined in terms of the parameters $m$, $r$ and $\beta$.  \rm

As the notation suggestions,  \eqref{eq:generating_series_one}
 and \eqref{eq:generating_series_three} are special cases of Theorem \ref{thm: main} when $r=0$,
after one employs classical formulas for Chebyshev polynomials in terms of trigonometric and inverse
trigonometric functions.
Similarly, \eqref{eq:generating_series_two} follows from Theorem \ref{thm: main} by taking $r=m/2$, which
is possible since it is assumed in this case that $m$ is even.

From Theorem \ref{thm: main} one can derive a recurrence formula for the coefficients
in the series expansion of $f_{m,r}(s,\beta)$.  More or less, if $P(s)$ is a convergent Taylor
series at $s=0$, and if we have that $P(s) = Q_{1}(s)/Q_{2}(s)$ where $Q_{1}(s)$ and $Q_{2}(s)$ are polynomials,
then one simply needs to equate the coefficents of $s$ in the expression $Q_{2}(s)P(s) = Q_{1}(s)$.
As it turns out in this case, there are convenient formulas for the series expansions of
the Chebyshev polynomials $T_{m}(z)$ and $U_{m}(z)$ at $z=1$; see \ref{appendix: Cheb pol}.
 From these computations, we
arrive at the following corollary.

\begin{corollary}\label{cor: recurrence even powers}
  Define the parameters $m$ and $r$ as above.
  Set the constants $a_{m}(j)$ and $b_{m}(j)$ as in equations \eqref{eq: coeff of T_n}
and \eqref{eq: coeff of U_n}, respectively.  For $\beta\notin \mathbb{Z}$ and
any integer $n \geq 0$, define the numbers
  $$
  c_{m,r}(\beta,n):=  e^{2\pi i \beta \ell /m} (-1)^{n}2^{-(n+1)} C_{m,r}(\beta,n+1).
  $$
Then we have the recurrence relation that
\begin{equation*}
\sum_{j=0}^{n}\binom{n}{j}\tilde{a}_m(n-j)c_{m,r}(\beta,j)=b_{m-\ell-1}(n)+e^{2\pi i \beta } b_{\ell-1}(n),
\end{equation*}
where $\tilde{a}_m(0)=1-\cos(2\pi \beta )$, $\tilde{a}_m(k)=a_m(k)$ for $k\geq 1$.

Similarly, when $\beta \in \mathbb{Z}$ and $n\geq 0$, define the numbers
\begin{equation}\label{eq: defn c m,r at n}
c_{m,r}(n):=  (-1)^{n}2^{-(n+1)} C_{m,r}(n+1).
\end{equation}
Then we have the recurrence relation that
\begin{equation}\label{eq: recurrence rel beta=0}
  \sum_{j=0}^{n}\binom{n}{j}a_m(j+1)c_{m,r}(n-j)=b_{m-\ell-1}(n+1)+ b_{\ell-1}(n+1) - \frac{a_m(n+2)}{m}.
\end{equation}
\end{corollary}

From Theorem \ref{thm: main} and Corollary \ref{cor: recurrence even powers} one can obtain an abundance of
specific formulas, each one of which can be described as mathematically appealing.  For example, we will show that for any $k\geq1$ one has that
\begin{equation}\label{eq: sum csc to 4}
 \sum_{j=1}^{3k-1}\csc^{4}\left(\frac{j\pi}{3k}\right)\cos\left(\frac{2\pi j}{3}\right)=-\frac{1}{45}\left(39k^4+30k^2+11\right)
\end{equation}
as well as that
\begin{equation}\label{eq: sum csc to 2 with 1/3}
\sum_{j=0}^{3k-1}\csc^2\left(\frac{2j+1}{6k}\pi\right)\omega^j=3k^2 e^{-\frac{i\pi}{3}}
\end{equation}
where $\omega$ is a primitive third root of unity.
The recursive formulas in Corollary \ref{cor: recurrence even powers} allow one to readily evaluate series with higher powers.
Again, these formulas are special evaluations of the above stated main Theorem.

\begin{remark} \rm
In \cite{Za96} Zagier proved a different recursion relation between certain cosecant sums.
Our formula is simpler in the sense that it is linear and whereas the formula in \cite{Za96} is quadratic.
Our formulas are thus analogous to linear recursion relations between zeta values like those
found from, for example, \cite{F16}, \cite{FK17}, \cite{Me17} and references therein.
\end{remark}

We shall now consider secant sums.
Let the parameters $m$ and $r$ be defined as above. For any real number $\alpha$ such that $\alpha  -m/2 \notin \mathbb{Z}$
define the generating function
\begin{equation*}
h_{m,r}(s,\alpha)= \sum_{n=0}^{\infty}S_{m,r}(\alpha,n+1)s^n
\end{equation*}
associated to the sequence of series \eqref{eq: sec sum at n - general}.  Additionally, define the generating function
\begin{equation}\label{eq: Taylor series of Sec2}
h_{m,r}(s)= \sum_{n=0}^{\infty} S_{m,r}(n+1)s^n
\end{equation}
associated to the sequence of secant sums \eqref{eq: non-twist sec sum at n - general}.
By taking $\beta=\alpha- m/2$ in Theorem \ref{thm: main}, we immediately deduce the following corollary.

\begin{corollary}
For all complex $s$ with $\vert s \vert$  sufficiently small, the series which defines $h_{m,r}(s,\alpha)$
converges uniformly and absolutely.  Furthermore, the function $h_{m,r}(s,\alpha)$ admits a meromorphic
continuation to all complex $s$, and we have that
$$
h_{m,r}(s,\alpha)= 2(-1)^{\ell}e^{-2\pi i \alpha \ell /m}\cdot \frac{U_{m-\ell-1}(1-2s)+(-1)^me^{2\pi i \alpha }U_{\ell-1}(1-2s)}{T_{m}(1-2s)-(-1)^m\cos2\pi \alpha }.
$$
As for \eqref{eq: Taylor series of Sec2}, there are two cases to consider.
If $m$ is odd, then $h_{m,r}(s) = h_{m,r}(s,0)$.  If $m$ is
even, then $S_{m,r}(n) = C_{m,r}(n)$; hence, the evaluation for \eqref{eq: Taylor series of Sec2} in
this case
is given by \eqref{eq: gen funct beta =0}.
\end{corollary}

\subsubsection{Generating functions for twisted sums at double arguments}

As we will show, the resolvent kernel $G_{X_m,\chi_{\beta}}(x,y;s)$ at $s=-1$ yields the
generating function the powers of secants and cosecants at double arguments. In particular,
see Section \ref{sec: sums double arg}, Theorem \ref{thm: main2} for our second main result,
which is the evaluation of the generating functions associated to the sequences
of the sums \eqref{eq: sec sum at even arg n - general}
and of the sums \eqref{eq: cosec sum at even arg n - general}.
As an application of Theorem \ref{thm: main2}, we obtain the succinct formulas that
\begin{equation}\label{eq: sec to 1}
\frac{1}{3k}\sum_{j=0}^{3k-1}\sec\left(\frac{4j}{3k} \pi\right)\omega^j=(-1)^{\frac{k-1}{2}}
\end{equation}
and
\begin{equation}\label{eq: sum sec to 2 with 1/3}
\frac{1}{3k}\sum_{j=0}^{3k-1}\sec^2\left(\frac{4j}{3k} \pi\right)\omega^j=-k
\end{equation}
where $\omega$ is a primitive third root of unity and $k \geq 1$.  As in the
previous section, we state and prove recursive relations for the sequences of
these sums.

\subsubsection{Evaluation of the Dirichlet $L$-function of  a cycle graph}

Let $m>1$ be an integer. The Dirichlet $L$-function of a cycle graph $X_m$ is the  \emph{spectral} $L$-function
corresponding to the spectrum of a combinatorial Laplacian.  Specifically, the function is
defined for any even Dirichlet character $\chi$ of modulus $m$ and any complex number $s$ by
\begin{equation}\label{eq: def of spectral L}
L_{X_m}(s,\chi)=\sum_{j=1}^{m-1}\chi(j)\csc^{2s}\left(\frac{j\pi}{m}\right);
\end{equation}
see \cite{F16, XZZ22}. For odd Dirichlet characters the similar sum is identically $0$.
However, the authors in \cite{XZZ22} propose a replacement.  Specifically, it is suggested that one
should consider the function
\begin{equation}\label{eq: def of spectral L odd}
\tilde{L}_{X_m}(s,\chi)=\sum_{j=1}^{m-1}\chi(j)\csc^{2s}\left(\frac{j\pi}{m}\right)\cot\left(\frac{j\pi}{m}\right).
\end{equation}

The functions \eqref{eq: def of spectral L} and \eqref{eq: def of spectral L odd} can be used to evaluate the classical Dirichlet $L$-functions at even and odd integers, respectively; see \cite{XZZ22}.  Hence, it is of interest to deduce an explicit evaluation of those functions.
In Section \ref{sec: mult twist} we will prove that for any even, primitive Dirichlet character $\chi$ one has
\begin{equation} \label{eq: evaluation of L at n}
L_{X_m}(n,\chi) = (-1)^{n+1} 2^{n} \frac{m}{ \overline{\tau(\chi)}} \sum_{r=0}^{m-1} \overline{\chi(r)} c_{m,r}(n-1), \quad n\in\mathbb{N},
 \end{equation}
where the coefficients $c_{m,r}(n-1)$ are explicitly computable for all positive integers $n$ when using the linear recurrence \eqref{eq: recurrence rel beta=0}.

In summary, from Theorem \ref{thm: main} and Corollary \ref{cor: recurrence even powers} one has a  method by which
\eqref{eq: evaluation of L at n} is explicitly computable in terms of coefficients of Chebyshev polynomials.
The main theorem in \cite{XZZ22} proves a relation involving the values of the Dirichlet $L$-functions at positive integers
in terms of the values \eqref{eq: evaluation of L at n}; see Theorem A of \cite{XZZ22}.  In Section 5 of
\cite{XZZ22} the authors posed the question of determining a direct way by which one can evaluate
\eqref{eq: evaluation of L at n}, so then one can evaluate Dirichlet $L$-functions.  Our results from
Section \ref{sec: mult twist} answer this question as stated in \cite{XZZ22}.

An explicit expression for values of \eqref{eq: def of spectral L odd} can be proved by differentiating
the shifted $L$-function with respect to  $\beta$.  This computation is described in Section \ref{sec: dif int resp to beta}.

\subsection{Organization of the article}  In the next section we recall material from
the literature regarding the continuous time heat kernel on a Cayley graph.  As stated, for this
paper the Cayley graph we consider is associated to $\mathbb{Z}/m\mathbb{Z}$, which
is the group of integers modulo $m$ with edges given by connecting an edge to its
two nearest neighbors.  In Section 3 we define and study the corresponding resolvent
kernel, which amounts to the Laplace transform in the time variable of the heat kernel.
In Section 4 we prove the main results as stated above, and in Section 5 we develop
further general results associated to secant and cosecant sums with doubled arguments.
In Section 6 we answer the aforementioned question posed in \cite{XZZ22} which involves certain special values
of spectral $L$-functions with a Dirichlet character.  Finally, in Section 7, we present a
few concluding remarks which suggest further studies which could be undertaken
based on the results and methods presented in this article.

\section{Heat kernel on Cayley graphs}

\subsection{Weighted Cayley graphs of abelian groups}\label{sec: weighted cayley graphs}
Let $G$ be a finite or countably infinite abelian group with composition law which is written additively.
Let $S\subseteq G$ be a finite symmetric subset of $G$. The symmetry condition means that if $s\in S$ then $-s\in S$.

Let $\alpha \colon S \to \mathbb{R}_{>0}$ be a function such that $\alpha(s) = \alpha(-s)$. The
weighted and undirected Cayley graph $X= \mathcal{C}(G, S, \alpha)$ of $G$ with respect to $S$ and $\alpha$ is constructed as follows. The vertices
of $X$ are the elements of $G$, and two vertices $x$ and $y$  are connected with an edge if and only if $x-y \in S$.  The weight
$w(x,y)$ of the edge $(x,y)$ is defined to be $w(x,y) \colonequals \alpha(x-y)$.
One can show that $X$ is a regular graph of degree
\[d= \sum_{s \in S} \alpha (s).\]
If $\alpha$ is a probability distribution on $S$, then the degree of the graph $X$ equals $1$. In this case we will denote $\alpha$ by $\pi_S$.

A function $f: G\to \mathbb{C}$ is an $L^2$-function if $\sum_{x\in G} |f(x)|^2 <\infty$. The set of $L^2$-functions on $G$
is a Hilbert space $L^2(G, \mathbb{C})$ with respect to the classical scalar product of functions
\[ \langle f_1,f_2 \rangle = \sum_{x\in G} f_1(x)\overline{f_2(x)}.\]

We will denote by $\delta_x$ the standard delta function, meaning $\delta_x(x) = 1$ and $\delta_x(y) = 0$ for $x \neq y$.

The adjacency operator  $\mathcal{A}_X \colon L^2(G, \mathbb{C}) \to L^2(G, \mathbb{C})$ of the graph $X$ is defined as
$$
(\mathcal{A}_X f)(x) = \sum_{x-y \in S} \alpha(x-y)f(y).
$$
When $X$ is finite, the adjacency operator when written with respect to the standard basis is called the adjacency
matrix $A_X$ of the graph $X$. The $(x,y)$-entry of the adjacency matrix is $A_X(x,y) = \alpha(x-y)$.
Since $\alpha(x-y)=\alpha(y-x)$, the matrix $A_X$ is symmetric.
Moreover, when $\alpha = \pi_S$, $A_X$ has the property that the elements in any column,
or any row, sum up to one.

Given $x$ in the finite abelian group $G$, let $\chi_x$ denote the character of $G$ corresponding to $x$ in a chosen isomorphisms between $G$ and its dual group; see, for example, \cite{CR62}.
As proved in Corollary 3.2 of \cite{Ba79},
the character $\chi_x$ is an eigenfunction of the adjacency operator $\mathcal{A}_X$ of $X$ with corresponding eigenvalue
\begin{equation*}
\eta_{x}=\sum_{s\in S} \alpha(s) \chi_{x}(s).
\end{equation*}

\subsection{Heat kernel on weighted Cayley graphs}

Let $X$ denote the weighted Cayley graph $\mathcal{C}(G,S,\pi_S)$. The standard, or random walk,
Laplacian $\Delta_X$ is defined to be the operator on $L^2(G, \mathbb{C})$ given by
\[\Delta_X f(x) = f(x)-\sum_{x-y \in S} \pi_S(x-y)f(y).
\]

The heat kernel $K_X: G\times G \times \mathbb{R}_{\geq 0} \to \mathbb{R}$ on $X$ is defined to be a solution to the equation
\begin{equation} \label{eq: heat eq}
(\partial_{t} + \Delta_X) K_X(x,y;t)= 0
\,\,\,\,\,
\text{\rm for  $t> 0$,}
\end{equation}
when viewed as a function of $x\in G$ for a fixed $y\in G$, and with initial condition
\begin{equation} \label{eq: heat eq_int}
\lim_{t\downarrow 0}K_X(x,y;t)=\delta_{x}(y).
\end{equation}
It can be shown that \eqref{eq: heat eq} and \eqref{eq: heat eq_int} also holds if we interchange the roles of $x$ and $y$.

When the graph X is countable with bounded vertex degree, it is shown in \cite{Do06} and
\cite{DM06} that the continuous time heat kernel exists and is unique among all bounded functions.

\subsection{Twisted heat kernel on $\mathbb{Z}/m\mathbb{Z}$}

Let $G=\mathbb{Z}$, and consider the Cayley graph $X=\mathcal{C}(G,S,\pi_S)$ when $S=\{-1,1\}$ and with $\pi_S(1)=\pi_S(-1)=1/2$. Then an
elementary computation involving properties of the $I$-Bessel function shows that the heat kernel on $X$ is given by
$$
K_{X}(x,y;t)=e^{-t}I_{x-y}(t);
$$
see section 3 of \cite{KN06}.
In subsequent computations, we will use that $I_\nu(t)=I_{-\nu}(t)$ for any $\nu\in\mathbb{N}$.
For an explicit solution of a more general type of diffusion equation on $X$, we refer the interested
reader to  \cite{SS14} and \cite{SS15}.

Let $m>1$ be a positive integer, and let $G_m=\mathbb{Z}/m\mathbb{Z}$ be the cyclic group of order $m$ with addition modulo $m$.
Denote by $X_{m}$ the Cayley graph $\mathcal{C}(G_m,S,\pi_S)$ where $S=\{-1,1\}$ and $\pi_S(1)=\pi_S(-1)=1/2$;
in case $m=2$ then $X_{2}$ has two edges.

For $\beta\in[0,1)$, $\chi_\beta(x)\colonequals \exp(2\pi i \beta x)$ is an additive character of $\mathbb{Z}$. The \it $\chi_{\beta}$-twisted heat kernel \rm on the Cayley graph $X_{m}$  is defined to be a function
\begin{equation}\label{eq:beta_twisted_HK}
K_{X_{m},\chi_{\beta}}(x,y;t): G_m\times G_m \times \mathbb{R}_{\geq 0} \to \mathbb{R},
\end{equation}
and it has the following properties.
For a fixed $y\in G_m$, and viewed as a function of $x$, \eqref{eq:beta_twisted_HK} satisfies the transformation property
\begin{equation}\label{eq: transf property twist}
K_{X_{m},\chi_{\beta}}(x+km,y;t)=\chi_\beta(k)K_{X_{m},\chi_{\beta}}(x,y;t),\quad \text{for all  } k\in\mathbb{Z}.
\end{equation}
Similarly, one has the analogue of \eqref{eq: transf property twist} when the heat kernel is viewed as a function of $y$ for a fixed $x\in G_m$
after replacing $\chi_\beta$ by its complex conjugate.
Additionally, when viewed as a function of  $t$, \eqref{eq:beta_twisted_HK} satisfies the heat equation \eqref{eq: heat eq} with the initial condition $\lim_{t\downarrow 0}K_{X_{m},\chi_{\beta}}(x,y;t)=\delta_{x}(y)$.

Using the method of images, as in \cite{KN06}, \cite{Do12} and \cite{CHJSV23}, one has the following expression for the twisted heat kernel $K_{X_{m},\chi_{\beta}}(x,y;t)$. 
\begin{lemma}
  With the notation as above, the twisted heat kernel $K_{X_{m},\chi_{\beta}}(x,y;t)$ is given by
\begin{equation} \label{eq: explicit m,b,beta twisted HK}
K_{X_{m},\chi_{\beta}}(x,y;t) = \sum_{k\in \mathbb{Z}} e^{-2\pi i \beta k} e^{-t} I_{x-y+km}(t).
\end{equation}
\end{lemma}
\begin{proof}
  First, we observe that the series on the right-hand side of \eqref{eq: explicit m,b,beta twisted HK} converges
  uniformly and absolutely for all $t\geq 0$, due to the property that $I_\nu(t)=I_{-\nu}(t)$ for $\nu\in\mathbb{N}$ and the bound
\begin{equation}\label{eq: I bessel bound}
\sum_{k=0}^{\infty}|I_{x+km}(t)|\leq e^t
\end{equation}
which is valid for all (fixed) integers $x$; see \cite{KN06}, section 5. The transformation property \eqref{eq: transf property twist} follows from the definition \eqref{eq: explicit m,b,beta twisted HK}.  Namely, for any $\ell\in\mathbb{Z}$ we have, by a substitution $j=k+\ell$, that
\begin{align*}
K_{X_{m},\chi_{\beta}}(x+\ell m,y;t)&= \sum_{k\in \mathbb{Z}} e^{-2\pi i \beta k} e^{-t} I_{x-y+(k+\ell)m}(t)\\ &= \sum_{j\in \mathbb{Z}} e^{-2\pi i \beta (j-\ell)} e^{-t} I_{x-y+jm}(t)\\&= e^{2\pi i \beta\ell}K_{X_{m},\chi_{\beta}}(x,y;t).
\end{align*}
Finally, we have that
$e^{-t}I_{x-y+km}(t)$ satisfies the equation
$$
\partial_t(e^{-t}I_{x-y+km}(t))= - \left(e^{-t}I_{x-y+km}(t) - \frac{1}{2}\left(e^{-t}I_{x-y+km +1}(t) + e^{-t}I_{x-y+km-1}(t)\right)\right),
$$
for all $k\in\mathbb{Z}$.
With all this, we conclude that \eqref{eq: explicit m,b,beta twisted HK} is indeed the heat kernel on $X_{m}$ twisted by $\chi_\beta$.
\end{proof}

We can reformulate the lemma to give a slightly different expression for the twisted heat kernel $K_{X_{m},\chi_{\beta}}$ that is more suitable for our purposes.
\begin{lemma}
With the notation as above, let $\ell\in\{0,\ldots,m-1\}$ be such that $\ell\equiv\, (x-y)\, (\mathrm{mod}\, m)$. Then
\begin{equation} \label{eq: explicit reduced m,b,beta twisted HK}
K_{X_{m},\chi_{\beta}}(x,y;t) = e^{-2\pi i \beta \frac{\ell - (x-y)}{m}} \sum_{j=-\infty}^{\infty} e^{-2\pi i \beta  j} e^{-t} I_{\ell+jm}(t).
\end{equation}
\end{lemma}

The twisted heat kernel on $X_{m}$ has a spectral expansion in terms of eigenfunctions and eigenvalues of the Laplacian $\Delta_{X_{m}}$. Namely, the eigenfunctions $\{\psi_j\}_{j=0}^{m-1}$ are given in terms of the normalized twisted characters, meaning that
\begin{equation}\label{psi j}
\psi_j(x)=\frac{1}{\sqrt{m}}\exp\left(2\pi i \frac{j+\beta}{m} x\right)
\,\,\,
\text{\rm for}
\,\,\,
x\in G_m
\,\,\,
\text{\rm and}
\,\,\,
j=0,\ldots,m-1.
\end{equation}
The normalization is chosen so that the $L^2$-norm of $\psi_j(x)$ on $G_m$ equals one.
The eigenvalues are described in section \ref{sec: weighted cayley graphs} for the adjacency operator,
which gives that
\begin{equation}\label{eq: lambda j}
\lambda_j= 1-\frac{1}{2}\left(\exp\left(2\pi i \frac{j+\beta}{m} \right) +\exp\left(-2\pi i \frac{j+\beta}{m}\right)\right)=  2\sin^2\left(\pi\frac{(j+\beta)}{m} \right)
\end{equation}
for $j=0,\ldots,m-1.$
With this notation, the spectral expansion of $K_{X_{m},\chi_{\beta}}(x,y;t)$ is given by
\begin{equation}\label{eq: HK spectral exp}
K_{X_{m},\chi_{\beta}}(x,y;t)= \sum_{j=0}^{m}e^{-\lambda_j t} \psi_j(x)\overline{\psi_j (y)}
\,\,\,
\text{\rm for}
\,\,\,
 x,y\in G_m
 \,\,\,
\text{\rm and}
\,\,\,
t\geq0.
\end{equation}
This identity can, of course, also be verified directly.

\section{Twisted resolvent kernel on $X_{m}$}

In this section we compute the twisted resolvent kernel, meaning the Green's function on $X_{m}$; see \cite{CY00} for related
results on the certain
graphs which require that the
eigenvalues are non-zero and the additive shift $\beta=0$.
Note that throughout this paper $\sqrt{s}$ denotes the principal branch of the square-root.

Our starting point in computing the twisted resolvent kernel on $X_{m}$ is the spectral expansion \eqref{eq: HK spectral exp}. For a complex number $s$ with $\mathrm{Re}(s)>0$, the resolvent kernel, or Green's function, is defined as
\begin{equation}\label{eq: green's functioin from HK}
G_{X_{m},\chi_{\beta}}(x,y;s)\colonequals \int\limits_{0}^{\infty} e^{-st} K_{X_{m},\chi_{\beta}}(x,y;t)dt.
\end{equation}
Since the heat kernel is well defined and bounded for all $t\geq 0$, the integral
in \eqref{eq: green's functioin from HK} converges and defines a holomorphic function of $s$ in the half-plane $\mathrm{Re}(s)>0$.

With all this, we have the following evaluation of the resolvent kernel \eqref{eq: green's functioin from HK}.

\begin{proposition}
  With the notation as above, write $x-y\equiv\ell\in \{0,\ldots,m-1\}$.  Then for $s\in\mathbb{C}$ with $\mathrm{Re}(s)>0$ we have that
\begin{align}\label{eq: Geen final f-la}\nonumber
G_{X_{m},\chi_{\beta}}(x,y;s)&= \frac{e^{-2\pi i \beta \frac{\ell - (x-y)}{m}}}{\sqrt{s^2+2s}} \\& \cdot \frac{\sinh\left(m-\ell \right)\cosh^{-1}(s+1)+e^{2\pi i \beta }\sinh\left( \ell\cosh^{-1}(s+1)\right)}{\cosh\left(m\cosh^{-1}(s+1)\right)-\cos2\pi \beta }.
\end{align}
\end{proposition}

\begin{proof}
We begin with \eqref{eq: explicit reduced m,b,beta twisted HK}. From the bound \eqref{eq: I bessel bound}, it is evident that for $s\in\mathbb{C}$ with $\mathrm{Re}(s)>0$  that the series
$$
\sum_{j=-\infty}^{\infty} e^{-2\pi i \beta j} e^{-(s+1)t} I_{\ell+jm}(t)= e^{-st} K_{X_{m},\chi_{\beta}}(x,y;s)
$$
can be integrated as in \eqref{eq: green's functioin from HK} term by term.  When computing these
integrals, we get the expression that
\begin{equation}\label{eq: green f-on sum of int}
G_{X_{m},\chi_{\beta}}(x,y;s)= e^{-2\pi i \beta \frac{\ell - (x-y)}{m}} \sum_{j=-\infty}^{\infty} e^{-2\pi i \beta  j}\int\limits_{0}^{\infty}e^{-(s+1)t} I_{|\ell+jm|}(t)dt,
\end{equation}
where, as stated above,  we have used that $I_{\nu}(t)=I_{-\nu}(t)$ for any integer $\nu$.

The integral \eqref{eq: green f-on sum of int} is the Laplace transform of the $I$-Bessel function. Hence, we can apply \cite{GR07},
formula 109 on p. 1116 with $\nu=|\ell+jm|\geq 0$ and $a=1$; note that the variable $s$ in this formula from \cite{GR07} is  our $s+1$.
The assumption from \cite{GR07} that $\mathrm{Re}(s+1)>a=1$ is fulfilled for $s\in\mathbb{C}$ with $\mathrm{Re}(s)>0$.  So then, we have that
\begin{equation}\label{eq: laplace transf I Bessel}
\int\limits_{0}^{\infty}e^{-(s+1)t} I_{|\ell+jm|}(t)dt = \frac{1}{\sqrt{s^2+2s}}\left(s+1-\sqrt{(s+1)^2 -1}\right)^{|\ell+jm|}.
\end{equation}
Since $\ell\in\{0,\ldots,m-1\}$, it is immediate that $|\ell+jm|=\ell +jm$  for all $j\geq 0$.  Also, we have that $|\ell+jm|=-\ell -jm$ for $j<0$.
Moreover, for real $s>0$, one has that $\left|s+1-\sqrt{(s+1)^2 -1} \right|<1$. Let $u= (s+1-\sqrt{(s+1)^2 -1})^{-1}>1$.  Therefore,
\begin{align*}
\sum_{j=-\infty}^{\infty} e^{-2\pi i \beta  j}&\left(s+1-\sqrt{(s+1)^2 -1}\right)^{|\ell+jm|}= \sum_{j=-\infty}^{\infty}
e^{-2\pi i \beta  j}u^{-|\ell+jm|} \\&= u^{-\ell}\sum_{j=0}^{\infty}\left( e^{-2\pi i \beta  }u^{-m}\right)^j +
u^{\ell}\sum_{j=1}^{\infty}\left( e^{2\pi i \beta  }u^{-m}\right)^j \\&=\frac{u^{m-\ell}- u^{-(m-\ell)} + e^{2\pi i \beta }(u^{\ell} - u^{-\ell})}{u^{-m} +u^{m} - 2\cos(2\pi \beta )}.
\end{align*}
Using that
$$
\exp(\cosh^{-1}(s+1))= s+1+\sqrt{(s+1)^2 -1} = (s+1-\sqrt{(s+1)^2 -1})^{-1}=u,
$$
we get that
\begin{align}\label{eq:series_sum}\nonumber
\sum_{j=-\infty}^{\infty}& e^{-2\pi i \beta b j}\left(s+1-\sqrt{(s+1)^2 -1}\right)^{|\ell+jm|}
\\&= \frac{\sinh\left((m-\ell)\cosh^{-1}(s+1)\right)+e^{2\pi i \beta}\sinh\left(\ell\cosh^{-1}(s+1)\right)}
{\cosh\left(m\cosh^{-1}(s+1)\right)-\cos2\pi \beta }.
\end{align}
When combining \eqref{eq:series_sum} with \eqref{eq: green f-on sum of int} and \eqref{eq: laplace transf I Bessel},
the proof of equation \eqref{eq: Geen final f-la} is completed for real and positive $s$.  Since the function on the right-hand side
of \eqref{eq: Geen final f-la} is holomorphic for $\mathrm{Re}(s)>0$, the proof for such $s$ follows from the principle of analytic continuation.
\end{proof}

We now will show that for $\beta \notin \mathbb{Z}$ the function on the right-hand side of \eqref{eq: Geen final f-la}
is holomorphic at $s=0$.

\begin{lemma} \label{lem: holom at s=0}
   For any $\ell\in\{0,\ldots, m-1\}$ and real number $\beta $ with $\beta \notin \mathbb{Z}$, the function
$$
g_{m,\ell}(s,\beta)= \frac{1}{\sqrt{(s+1)^2-1}}\cdot \frac{\sinh\left((m-\ell)\cosh^{-1}(s+1)\right)+e^{2\pi i \beta }\sinh\left(\ell\cosh^{-1}(s+1)\right)}{\cosh\left(m\cosh^{-1}(s+1)\right)-\cos2\pi \beta }
$$
is holomorphic at $s=0$.
\end{lemma}

\begin{proof}
  Since $g_{m,\ell}(s,\beta)$ is holomorphic in the half-plane $\mathrm{Re}(s)>0$, it suffices to show that
  $g_{m,\ell}(s,\beta)$ is bounded as $s\to 0$.
Indeed, for any positive integer $j$, it is elementary that
\begin{align*}
\sinh\left(j\cosh^{-1}(s+1)\right)&=\frac{1}{2}\left( (s+1+\sqrt{s^2+2s})^j -  (s+1-\sqrt{s^2+2s})^j\right)
\\&=j\sqrt{s^2+2s}+O(s)
\,\,\,\,\,
\text{\rm as $s \to 0$.}
\end{align*}
For $\beta \notin \mathbb{Z}$,  $\cos(2\pi \beta ) \neq 1$, so then
$$
\lim_{s\to 0} g_{m,\ell}(s,\beta)= \frac{(m-\ell)+\ell e^{2\pi i \beta }}{1-\cos2\pi \beta }.
$$
\end{proof}

With the spectral expansion \eqref{eq: HK spectral exp} of the heat kernel, we get another expression for the resolvent kernel
upon integrating as in \eqref{eq: green's functioin from HK}.  Specifically, we have that
$$
G_{X_{m},\chi_{\beta}}(x,y;s)=\sum_{j=0}^{m-1}\frac{1}{s+\lambda_j} \psi_j(x)\overline{\psi_j (y)}
\,\,\,\,\,
\text{\rm for}
\,\,\,\,\,
 \mathrm{Re}(s)>0.
$$
From the formulas \eqref{psi j} and \eqref{eq: lambda j} for $\psi_j$ and $\lambda_j$, we arrive at the expression that
\begin{equation}\label{eq: green's functioin from spect exp}
G_{X_{m},\chi_{\beta}}(x,y;s)=\frac{1}{m}\sum_{j=0}^{m-1} \frac{1}{s+2\sin^2\left(\pi\frac{j+\beta}{m} \right)}
\exp\left(2\pi i \frac{j+\beta}{m}( x-y)\right).
\end{equation}
It is immediate that the right-hand-side of \eqref{eq: green's functioin from spect exp} is a meromorphic
function with simple poles whenever $s$ is one of the finite points for which $s = -2\sin^2\left(\pi\frac{j+\beta}{m}\right)$.
In effect, our main results follow from the identity obtained by equating \eqref{eq: Geen final f-la}
and \eqref{eq: green's functioin from spect exp}.

\section{Proof of Theorem \ref{thm: main}}

We start by proving the first part of Theorem \ref{thm: main}.  Assume $\beta\notin \mathbb{Z}$.
As stated, for $x,y\in X_m$ and $s\in \mathbb{C}$ with $\mathrm{Re}(s)> 0$, we have two expressions
\eqref{eq: green's functioin from spect exp} and  \eqref{eq: Geen final f-la}
for the Green's function $G_{X_{m},\chi_{\beta}}(x,y;s)$.  Therefore, the right-hand sides of those
formulas are equal. Set $r=(x-y)$, and $\ell$ as before.  With this, we get, upon cancelling a factor $\exp(2\pi i \beta r/m)$,
the identity that
\begin{align}\label{eq: identity for gen function}\nonumber
\frac{1}{m}\sum_{j=0}^{m-1}& \frac{e^{2\pi i \frac{jr}{m}}}{s+2\sin^2\left(\pi\frac{(j+\beta)}{m} \right)} \\&= \frac{e^{-2\pi i \beta \frac{\ell}{m}}}{\sqrt{s^2+2s}} \frac{\sinh\left((m-\ell)\cosh^{-1}(s+1)\right)+e^{2\pi i \beta }\sinh\left(\ell\cosh^{-1}(s+1)\right)}{\cosh\left(m\cosh^{-1}(s+1)\right)-\cos2\pi \beta }.
\end{align}
From the definition of the Chebyshev polynomials of the first and the second kind, we have for $\mathrm{Re}(s+1)>1$ that
$$
T_m(s+1)= \cosh\left(m\cosh^{-1}(s+1)\right) \quad \text{and} \quad U_{m-1}(s+1)=\frac{\sinh\left(m\cosh^{-1}(s+1)\right)}{\sqrt{(s+1)^2-1}}.
$$

Let
\begin{equation}\label{eq: defn of F}
F_{m,r}(s,\beta)= e^{-2\pi i \beta \ell /m}\cdot \frac{U_{m-\ell-1}(s+1)+e^{2\pi i \beta }U_{\ell-1}(s+1)}{T_{m}(s+1)-\cos2\pi \beta }.
\end{equation}
Then, for $\mathrm{Re}(s+1)>1$ we have that
$$
F_{m,r}(s,\beta) = \frac{e^{-2\pi i \beta \frac{\ell}{m}}}{\sqrt{s^2+2s}} \frac{\sinh\left((m-\ell)\cosh^{-1}(s+1)\right)+e^{2\pi i \beta }\sinh\left(\ell\cosh^{-1}(s+1)\right)}{\cosh\left(m\cosh^{-1}(s+1)\right)-\cos2\pi \beta }
$$
and, by using \eqref{eq: identity for gen function},
\begin{equation}\label{eq:resolvent_formula_final}
F_{m,r}(s,\beta) = \frac{1}{m}\sum_{j=0}^{m-1} \frac{e^{2\pi i \frac{jr}{m}}}{s+2\sin^2\left(\pi\frac{(j+\beta)}{m} \right)}
\end{equation}
The equality \eqref{eq:resolvent_formula_final}, which holds for  $\mathrm{Re}(s+1)>1$, extends to an
equality of meromorphic functions which holds for all values of the complex variable $s$.
In particular,
for any fixed $\beta\in(0,1)$, Lemma \ref{lem: holom at s=0} yields that the function $F_{m,r}(s,\beta)$ is holomorphic at $s=0$. Moreover by differentiating the right-hand side of \eqref{eq:resolvent_formula_final} $n$ times with respect to $s$ evaluating at $s=0$ we get
\begin{equation*}
\partial_s^n\left.F_{m,r}(s,\beta)\right|_{s=0}=
(-1)^n n! 2^{-(n+1)} \cdot C_{m,r}(\beta,n+1).
\end{equation*}
This yields that
\begin{align*}
F_{m,r}(s,\beta) &= \sum\limits_{n=0}^{\infty} \partial_s^n\left.F_{m,r}(s,\beta)\right|_{s=0} \frac{s^{n}}{n!}
\\& =\sum\limits_{n=0}^{\infty} \left((-1)^n 2^{-(n+1)} \cdot C_{m,r}(\beta,n+1)\right)s^{n}
\end{align*}
for $s$ sufficiently close to zero.
This proves the first part of Theorem \ref{thm: main}, after the cosmetic change of variable for $s$ obtained by replacing $s$ with $-2s$.

To prove the second part, we notice that from \eqref{eq: identity for gen function} that one has the identity
\begin{equation}\label{eq:series_with_beta_equal_0}
F_{m,r}(s,\beta) - \frac{1}{m\left(s+2\sin^2\left(\pi\frac{\beta }{m} \right)\right)} =\frac{1}{m}\sum_{j=1}^{m-1} \frac{e^{2\pi i \frac{jr}{m}}}{s+2\sin^2\left(\pi\frac{j+\beta}{m} \right)}.
\end{equation}
For all  $s\in \mathbb{C}$ with $\mathrm{Re}(s)\geq 0$  the function on the right-hand side \eqref{eq:series_with_beta_equal_0} is continuous
at $\beta=0$ from the right.  Hence, the function on the left-hand side of \eqref{eq:series_with_beta_equal_0} must also be right-continuous,
so then we have that
\begin{equation}\label{eq: lim as beta to 0}
 \lim_{\beta\downarrow 0}\left( F_{m,r}(s,\beta) - \frac{1}{m \left(s+2\sin^{2}\left(\frac{\beta }{m}\pi\right)\right)} \right) = \frac{1}{m}\sum_{j=1}^{m-1} \frac{e^{2\pi i \frac{jr}{m}}}{s+2\sin^2\left(\pi\frac{j}{m} \right)}.
\end{equation}
Trivially, from \eqref{eq: defn of F} we obtain that
\begin{equation}\label{eq: defn F at beta is 0}
 \lim_{\beta\downarrow 0}\left( F_{m,r}(s,\beta) - \frac{1}{m \left(s+2\sin^{2}\left(\frac{\beta }{m}\pi\right)\right)} \right)= \frac{U_{m-\ell-1}(s+1)+U_{\ell-1}(s+1)}{T_{m}(s+1)-1}-\frac{1}{ms}=F_{m,r}(s).
\end{equation}

The function on the right-hand side of \eqref{eq: lim as beta to 0} is holomorphic at $s=0$.  Therefore,
 $F_{m,r}(s)$ is also holomorphic at $s=0$, and then
\begin{equation}\label{eq: deriv of F}
\partial_s^n\left.F_{m,r}(s)\right|_{s=0}=(-1)^n n! 2^{-(n+1)} \cdot C_{m,r}(n+1).
\end{equation}
This proves the second claim of Theorem \ref{thm: main}, again after replacing $s$ with $-2s$.

\begin{example}\rm
Consider any positive $m$, $\beta=1/2$, $r=0$ and $n=1$.  Then by taking $s=0$ in Theorem \ref{thm: main}, we get that
$$
\frac{1}{2}C_{m,0}(1/2,1)=f_{m,0}(0,1/2)
$$
or
$$
\frac{1}{2m} \sum_{j=0}^{m-1}\csc^2\left(\frac{2j+1}{2m}\pi\right)=\frac{U_{m-1}(1)}{T_m(1)+1}= \frac{m}{2}.
$$
This yields the well-known evaluation that
$$
\sum_{j=0}^{m-1}\csc^2\left(\frac{2j+1}{2m}\pi\right)=m^2;
$$
see \cite[Corollary 2.6]{BY02} and references therein regarding the appearance of those sums elsewhere in the literature.
\end{example}

\begin{example}\rm
 For any positive integer $k$, let  $m=3k$.  Take $\beta=1/2$ and $r=k$.
 Let $\omega$ denote the third root of unity. Then, for all positive integers $n$ one has the identity that
$$
\frac{1}{3k} \sum_{j=0}^{3k-1}\csc^{2n}\left(\frac{2j+1}{6k}\pi\right)\omega^j=(-1)^{n-1}2^n \partial_s^{n-1}\left. F_{3k,3}(s,1/2)\right|_{s=0},
$$
where
$$
 F_{3k,3}(s,1/2)=e^{-\frac{i\pi}{3}}\frac{U_{2k-1}(s+1)-U_{k-1}(s+1)}{T_{3k}(s+1)+1}.
$$
When $n=1$ this yields the formula \eqref{eq: sum csc to 2 with 1/3}.
For $n\geq 1$, one can use the expansions of $U_{2k-1}(z)$, $U_{k-1}(z)$ and $T_{3k}(z)$ at $z=1$, as provided in Section \ref{appendix: Cheb pol} below,
to get further evaluations.  For example, one gets that
$$
\sum_{j=0}^{3k-1}\csc^4\left(\frac{2j+1}{6k}\pi\right)\omega^j=k^2(13k^2+2) e^{-\frac{i\pi}{3}}.
$$

If one takes $\beta=0$ and the same values of $m$ and $r$, one gets the formula that
$$
\frac{1}{3k} \sum_{j=1}^{3k-1}\csc^{2n}\left(\frac{j\pi}{3k}\right)\omega^j=(-1)^{n-1}2^n \partial_s^{n-1}\left. F_{3k,3}(s)\right|_{s=0}
$$
where
$$
F_{3k,3}(s)=\frac{U_{2k-1}(s+1)+U_{k-1}(s+1)}{T_{3k}(s+1)-1} - \frac{1}{3ks}.
$$
By using the recurrence formula \eqref{eq: recurrence rel beta=0} with $n=0$, when
combined with evaluations \eqref{eq: coeff of T_n} and \eqref{eq: coeff of U_n}, one immediately derives the identity that
$$
 \sum_{j=1}^{3k-1}\csc^{2}\left(\frac{j\pi}{3k}\right)\cos\left(\frac{2\pi j}{3}\right)=-k^2-\frac{1}{3}.
$$
The recurrence formula \eqref{eq: recurrence rel beta=0} with $n=1$, combined with \eqref{eq: coeff of T_n} and \eqref{eq: coeff of U_n} below,
yields \eqref{eq: sum csc to 4}.
\end{example}

\section{Secant and cosecant sums of a double argument}\label{sec: sums double arg}

In this section we will study the resolvent kernel $G_{X_m,\chi_{\alpha}}(x,y;s)$, which equals
$F_{m,r}(s,\alpha)$ for $r=x-y$, in the neighbourhood of $s=-1$.  In doing so, we will prove the following theorem.

\begin{theorem}\label{thm: main2}
  Let $m\geq 1$ and $r$ be integers.  Let $\ell\in\{0,\ldots, m-1\}$ be such that $r\equiv \ell \,(\mathrm{mod}\, m)$. Let $\alpha$ be a real number
  such that $\alpha \notin \mathbb{Z}$ when $m\equiv 0 \,(\mathrm{mod}\, 4)$, $\alpha
  \notin \mathbb{Z}+\frac{1}{2}$ when $m\equiv 2 \,(\mathrm{mod}\, 4)$ and  $2\alpha
  \notin \mathbb{Z}+\frac{1}{2}$ when $m$ is odd. Then the generating function
\begin{equation}\label{eq: Taylor series s=-1}
\tilde{f}_{m,r}(z,\alpha)=- \sum_{n=0}^{\infty}\tilde{S}_{m,r}(\alpha,n+1)z^n
\end{equation}
for the sum \eqref{eq: sec sum at even arg n - general} is, for all complex $z$ in a neighbourhood of $z=0$, equal to
$$
\tilde{f}_{m,r}(z,\alpha)= e^{-2\pi i \alpha \frac{\ell}{m}}\cdot \frac{U_{m-\ell-1}(z)+e^{2\pi i \alpha }U_{\ell-1}(z)}{T_{m}(z)-\cos2\pi \alpha }=F_{m,r}(z-1,\alpha),
$$
where as above $U_{-1}(x)\equiv 0$.  Therefore, $\tilde{f}_{m,r}(z,\alpha)$ has a meromorphic
continuation to all complex $z$.  Moreover, the coefficients
$$
\tilde{c}_{m,r}(\alpha,n):= -e^{2\pi i \alpha \frac{\ell}{m}}\tilde{S}_{m,r}(\alpha,n+1)
\,\,\,\,\,
\text{\rm with $n\geq 0$}
$$
satisfy the recursive relation that
\begin{equation}\label{eq: recursion sec any power sum}
\sum_{j=0}^{n-1}\binom{n}{j} t_m(n-j)\tilde{c}_{m,r}(\alpha,j) + (t_m(0)-\cos2\pi\alpha )\tilde{c}_{m,r}(\alpha,n) =u_{m-\ell-1}(n)+e^{2\pi i \alpha } u_{\ell-1}(n),
\end{equation}
where $t_n(k)$ and $u_n(k)$ are given for $0\leq k \leq n$ by  \eqref{eq: coeff of T_n at zero}, and  $t_n(k)=u_n(k)=0$ for $k>n$.
\end{theorem}

\begin{proof}
Our starting point is the equation
\begin{equation}\label{eq: starting for sec thm}
\frac{1}{m}\sum_{j=0}^{m-1} \frac{e^{2\pi i \frac{jr}{m}}}{s+2\sin^2\left(\pi\frac{(j+\alpha)}{m} \right)} = e^{-2\pi i \alpha \ell /m}\cdot \frac{U_{m-\ell-1}(s+1)+e^{2\pi i \alpha }U_{\ell-1}(s+1)}{T_{m}(s+1)-\cos2\pi \alpha }.
\end{equation}
Equation \eqref{eq: starting for sec thm}
stems from \eqref{eq: identity for gen function} and \eqref{eq: defn of F} with $\beta=\alpha$, which comes from two different ways to write $F_{m,r}(s,\alpha)$. For real values of $\alpha$ such that $2\alpha \notin\frac{m}{2} \mathbb{Z}$ when $m$ is even and for $2\alpha \notin \mathbb{Z}+\frac{1}{2}$ when $m$ is odd,
 it is obvious that the left-hand side of \eqref{eq: starting for sec thm} is analytic at $s=-1$.  Therefore,
  $F_{m,r}(s,\alpha)$ is analytic at $s=-1$, for the given values of $m$ and $\alpha$.

By differentiating the left-hand side of \eqref{eq: starting for sec thm} $n$ times with respect to $s$, we get, after applying the trigonometric identity $-1+2\sin^2x= -\cos 2x$, that
\begin{align*}
\partial_s^n\left.F_{m,r}(s,\alpha)\right|_{s=-1}&=(-1)^n n! \cdot \frac{1}{m}\sum_{j=0}^{m-1}\frac{e^{2 \pi i \frac{jr}{m}}}{\left(-1+2\sin^2\left(\pi\frac{(j+\alpha)}{m} \right)\right)^{n+1}} \\&= -n! \tilde{S}_{m,r}(\alpha,n+1).
\end{align*}
Therefore, equation \eqref{eq: Taylor series s=-1} holds for $\tilde{f}_{m,r}(z,\alpha)= F_{m,r}(z-1,\alpha)$.
As in previous discussion, the
recursion formula \eqref{eq: recursion sec any power sum} follows from the uniqueness of the Taylor series expansion.
\end{proof}

By letting $\alpha = \beta - \frac{m}{4}$ in the above theorem, and using that $\cos x= \sin(\pi/2+x)$, we arrive at the following corollary

\begin{corollary}
Let $m\geq 1$ and $r$ be integers.  Let $\ell\in\{0,\ldots, m-1\}$ be such that $r\equiv \ell \,(\mathrm{mod}\, m)$.
Let $\beta$ be a real number such that  $2\beta \notin\mathbb{Z}$ when $m$ is odd and $\beta \notin
\mathbb{Z}$ when $m$ is even. Then the generating function
\begin{equation*}
\tilde{h}_{m,r}(z,\beta)=- \sum_{n=0}^{\infty}\tilde{C}_{m,r}(\beta,n+1)z^n
\end{equation*}
for the sum \eqref{eq: cosec sum at even arg n - general} is given for all complex $z$ in a neighbourhood of $z=0$ by
$$
\tilde{h}_{m,r}(z,\beta)= e^{-2\pi i \beta \frac{\ell}{m}}\cdot \frac{U_{m-\ell-1}(z)+e^{2\pi i \beta }U_{\ell-1}(z)}{T_{m}(z)-\cos\pi \left( 2\beta  - \frac{m}{2}\right)},
$$
where as above $U_{-1}(x)\equiv 0$.  Furthermore, $\tilde{h}_{m,r}(z,\beta)$ admits a meromorphic continuation
to all complex $z$.
\end{corollary}

The generating function for the sum \eqref{eq: non-twist cosec sum double at n - general} is obtained
in a similar manner.  Namely, from \eqref{eq: starting for sec thm} and by taking $\beta=\alpha$ with $s=z-1$, we get that
$$
\frac{1}{m}\sum_{j=0}^{m-1}\frac{e^{\frac{2\pi i r}{m}j}}{z-\cos\left(\frac{2\pi (j+\beta) }{m}\right)}= e^{-2 \pi i \beta  \ell/m}\cdot\frac{U_{m-\ell-1}(z) + e^{2\pi i \beta }U_{\ell-1}(z)}{T_m(z) - \cos(2 \pi\beta )}
$$
in a certain vertical strip in the complex $z$-plane depending on parameters $\beta$ and $m$. This yields for $\beta=\frac{m}{4}$  the identity
that
\begin{equation}\label{eq:double_sine_series}
\frac{1}{m}\sum_{j=0}^{m-1}\frac{e^{\frac{2\pi i r}{m}j}}{z+\sin\left(\frac{2\pi j }{m}\right)}= e^{-i\pi \ell/2}\cdot\frac{U_{m-\ell-1}(z) + e^{i\pi m/2}U_{\ell-1}(z)}{T_m(z) - \cos(m\pi/2)}.
\end{equation}
where both sides of \eqref{eq:double_sine_series} are holomorphic for all complex $z$ with $0<\mathrm{Re}(z)<\delta$
when
$$
0 < \delta < \min\{ |\sin(2\pi j /m)| \,\,\text{\rm for} \,\, j\in\{1,\ldots,m-1\} \setminus \{j_m\}^{\ast}\}.
$$
When $m$ is odd, this gives
$$
\frac{1}{m}\sum_{j=1}^{m-1}\frac{e^{\frac{2\pi i r}{m}j}}{z+\sin\left(\frac{2\pi j }{m}\right)}= e^{-i\pi \ell/2}\cdot\frac{U_{m-\ell-1}(z) + e^{i\pi m/2}U_{\ell-1}(z)}{T_m(z)}-\frac{1}{mz},
$$
while for even $m$, we get
$$
\frac{1}{m}\sum_{j\in\{1,\ldots,m-1\}\setminus \{j_m\}^{\ast}}\frac{e^{\frac{2\pi i r}{m}j}}{z+\sin\left(\frac{2\pi j }{m}\right)}= e^{-i\pi \ell/2}\cdot\frac{U_{m-\ell-1}(z) + e^{i\pi m/2}U_{\ell-1}(z)}{T_m(z) - (-1)^{m/2}}-\frac{2}{m z}.
$$
The left-hand sides of the above two displayed equations are holomorphic functions at $z=0$, hence so are the right-hand sides.
Moreover, for even $m$ we have
$$
\partial_z^n \left.\left(\frac{1}{m}\sum_{j=1}^{m-1}\frac{e^{\frac{2\pi i r}{m}j}}{z+\sin\left(\frac{2\pi j }{m}\right)}\right)\right|_{z=0}=(-1)^n n! \tilde{C}_{m,r}(n+1)
$$
for any non-negative integer $n$. An analogous conclusion holds true for odd $m$.

With all this, we have proved the following corollary.

\begin{corollary}
Let $m\geq 1$ and $r$ be integers.  Let $\ell\in\{0,\ldots, m-1\}$ be such that $r\equiv \ell \,(\mathrm{mod}\, m)$.  Then the generating function
\begin{equation*}
\tilde{h}_{m,r}(z)= \sum_{n=0}^{\infty}(-1)^n\tilde{C}_{m,r}(n+1)z^n
\end{equation*}
for the sum \eqref{eq: non-twist cosec sum double at n - general} is given for all complex $z$ in a neighbourhood of $z=0$ by
$$
\tilde{h}_{m,r}(z)=e^{-i\pi \ell/2}\cdot\frac{U_{m-\ell-1}(z) + e^{i\pi m/2}U_{\ell-1}(z)}{T_m(z) - \cos(m\pi/2)} - \frac{\delta(m)}{mz} ,
$$
where $\delta(m)=1$ if $m$ is odd and $\delta(m)=2$ if $m$ is even.
Furthermore, $\tilde{h}_{m,r}(z)$ admits a meromorphic continuation
to all complex $z$.
\end{corollary}

\begin{example}\rm
For any positive and odd integer $k$, let $m=3k$, and set $r=k$. Let $\omega$ denote the third root of unity.
By taking $\alpha=1/2$, from \eqref{eq: recursion sec any power sum} with $n=0$ one gets that
$$
\frac{1}{3k}\sum_{j=0}^{3k-1}\sec\left(\frac{2j+1}{3k} \pi\right)\omega^j =(-1)^{\frac{k-1}{2}}e^{-\frac{i\pi}{3}} .
$$
Similarly, by setting $\alpha=0$ in \eqref{eq: recursion sec any power sum}, one deduces identities \eqref{eq: sec to 1} and \eqref{eq: sum sec to 2 with 1/3}
by considering $n=0$ and $n=1$.
\end{example}

\begin{remark}\rm
Let us note here that the secant and cosecant sums \eqref{eq: sec sum at even arg n - general} and \eqref{eq: cosec sum at even arg n - general} with double argument taken to an even power are closely related to sums \eqref{eq: sec sum at n - general} and \eqref{eq: cosec sum at n - general}. 
For even values $m=2k$, one has  $$\tilde{C}_{2k,r}(\beta,2n)=C_{k,r}(\beta,n)\left(\frac{1+(-1)^r}{2}\right).$$
Given that we have different generating functions, those relations yield further identities satisfied by functions $f$ and $\tilde{h}$
and their derivatives.

We studied both types of sums because there are instances when one sum cannot be reduced to another one,
such as when taking odd powers in \eqref{eq: sec sum at even arg n - general}
and \eqref{eq: cosec sum at even arg n - general} or odd $m$ in \eqref{eq: sec sum at n - general}
and \eqref{eq: cosec sum at n - general}.
\end{remark}

\section{Sums twisted by a multiplicative character}\label{sec: mult twist}

In this section we will relate the results in this article to that from \cite{F16, FK17} and \cite{XZZ22}.
In particular, we will prove formula \eqref{eq: evaluation of L at n} for evaluation of the special values of the spectral $L$-function
associated to the cycle graphs $X_{m}$ at positive integers, thus providing an answer to the question raised at the end of
\cite{XZZ22}.

More precisely, we will consider the generating function for the $L$-function defined on $X_{m}$ for any even
Dirichlet character $\chi$ of modulus $m$ and any complex number $s$.  This $L$-function is given in \eqref{eq: def of spectral L}.
When the character is trivial, $L_{X_m}(s,\chi)$ becomes the spectral zeta function $\zeta_{X_m}(s)$ on $X_m$
which was studied in \cite{FK17}.  Note that the special values of $\zeta_{X_m}(s)$ at positive integers $n$ is the
non-twisted cosecant sum $C_m(0,2n)$ as defined in \eqref{eq: cosec sum basic}.

The following corollary evaluates the generating function for the special values of $L_{X_m}(n+1,\chi)$
associated to a primitive Dirichlet character modulo $m$ and $n\geq 0$.

\begin{corollary}\label{cor: twist by even char}
Let $m>1$ be an integer and assume $\chi$ is a primitive Dirichlet character modulo $m$. The generating function
$$
F_{m,\chi}(s)= \sum_{n=0}^{\infty}(-1)^n 2^{-(n+1)} \overline{L_{X_m}(n+1,\chi)}s^n
$$
coverges for $\vert s \vert$ sufficiently small.  Furthermore, we have that
\begin{equation}\label{eq:L_series_generating_function}
F_{m,\chi}(s) := \frac{m}{\tau(\chi)} \sum_{r=0}^{m-1}\chi(r)\left( \frac{U_{m-r-1}(s+1)+U_{r-1}(s+1)}{T_{m}(s+1)-1}-\frac{1}{ms}\right),
\end{equation}
where $\tau(\chi)$ denotes the Gauss sum associated to the character $\chi$, and
the value of $F_{m,\chi}(s)$ at $s=0$ is obtained by taking the limit as $s\to 0$.  Furthermore, the
series $F_{m,\chi}(s)$ admits a meromorphic continuation to all complex $s$
and, by \eqref{eq:L_series_generating_function}, is equal to a rational function in $s$.
\end{corollary}

\begin{proof}
It suffices to relate $L_{X_m}(n+1,\chi)$ to the sum of twists of $C_{m,r}(n+1)$ and apply the second part of Theorem \ref{thm: main}.
Recall the identity
$$
\sum_{r=0}^{m-1}\chi(r)e^{\frac{2\pi i r}{m} j}=\overline{\chi(j)}\tau(\chi),
$$
which holds for primitive Dirichlet characters.  From this, we immediately deduce that

\begin{align*}
\sum_{r=0}^{m-1}\chi(r) C_{m,r}(n)&=\frac{1}{m}\sum_{r=0}^{m-1}\chi(r)\sum_{j=1}^{m-1}\csc^{2n}\left(\frac{j\pi}{m}\right) e^{\frac{2\pi i r}{m} j}\\&= \frac{\tau(\chi)}{m}\sum_{j=0}^{m-1}\overline{\chi(j)}\csc^{2n}\left(\frac{j\pi}{m}\right)\\&= \frac{\tau(\chi)}{m} \overline{L_{X_m}(n,\chi)}.
\end{align*}
Therefore, for $s$ in a neighborhood of $s=0$, we deduce from \eqref{eq: deriv of F} that
\begin{equation}\label{eq: generating for L}
\sum_{n=0}^{\infty}(-1)^n 2^{-(n+1)} \overline{L_{X_m}(n+1,\chi)}s^n = \frac{m}{\tau(\chi)} \sum_{r=0}^{m-1}\chi(r) F_{m,r}(s).
\end{equation}
By observing that $ F_{m,r}(s)$ is defined by \eqref{eq: defn F at beta is 0}, the proof is them complete.
\end{proof}

The proof of \eqref{eq: evaluation of L at n} now readily follows by conjugating
\eqref{eq: generating for L} and recalling that, according to \eqref{eq: defn F at beta is 0}
and  \eqref{eq: deriv of F}, one has that
$$
F_{m,r}(s)=\sum\limits_{n=0}^{\infty}c_{m,r}(n) s^n
$$
in a neighbourhood of $s=0$, where $c_{m,r}(n)$ is defined by \eqref{eq: defn c m,r at n}.
Note that the terms $c_{m,r}(n)$
satisfy the recurrence relation \eqref{eq: recurrence rel beta=0}.

\begin{example}\rm
 When $n=0$, then a simple computation shows that
$$c_{m,r}(0)=(m^2-6mr+6r^2-1)/(6m).$$ This yields an interesting evaluation of $L_{X_m}(1,\chi)$, namely that
$$
L_{X_m}(1,\chi)=\frac{2}{\overline{\tau(\chi)}}\sum_{r=0}^{m-1}\overline{\chi(r)}(r-m)r.
$$
When $n=1$, further calculations easily produce that $L_{X_m}(2,\chi)$ is given by
$$
L_{X_m}(2,\chi)=-\frac{2}{3\overline{\tau(\chi)}}\sum_{r=0}^{m-1}\overline{\chi(r)}(r-2m)(r-m)r(r+m).
$$
These two formulas suggest a general pattern. From the recurrence relation \eqref{eq: recurrence rel beta=0}, it is immediate
that $mc_{m,r}(n-1)$ is a polynomial of degree $2n$ in two variables $m$ and $r$.  Hence, $L_{X_m}(n,\chi)$ can be expressed as $$
L_{X_m}(n,\chi) = \sum_{r=0}^{m-1}\overline{\chi(r)}P_{2n}(r,m)
$$
for a certain explicitly computable polynomial $P_{2n}(r,m)$ of degree $2n$.
\end{example}


Using results of Section \ref{sec: sums double arg}, it is possible to deduce further evaluations of secant and cosecant sums
of double arguments twisted by multiplicative characters, thus complementing results of \cite{BZ04} and \cite{BBCZ05}. For example,
consider a positive integer $m$ which is not divisible by $4$ and a primitive Dirichlet character $\chi$ modulo $m$.  By
reasoning as in the proof of Corollary \ref{cor: twist by even char}, with the starting point being Theorem \ref{thm: main2} with $\alpha=0$,
one will deduce the evaluation of the $L$-function given by
$$
\hat{L}_{X_m}(w,\chi)=\sum_{j=1}^{m-1}\chi(j)\sec^{w}\left(\frac{2j\pi}{m}\right)
\,\,\,\,\,
\text{\rm whenever $w=n$ and $n$ is a positive integer.}
$$
From this, we have the following corollary.

\begin{corollary}
  Let $m>1$ be an integer not divisible by $4$, and assume that $\chi$ is a
  primitive Dirichlet character modulo $m$. The generating function
$$
\hat{F}_{m,\chi}(s)= \sum_{n=0}^{\infty}\hat{L}_{X_m}(n,\chi)s^{n}
$$
is given for all complex $s$ with $\vert s \vert$
sufficiently small by
\begin{equation*}
\hat{F}_{m,\chi}(s) := - \frac{m}{\tau(\overline{\chi})} \sum_{r=0}^{m-1}\overline{\chi(r)}\left( \frac{U_{m-r-1}(s)+U_{r-1}(s)}{T_{m}(s)-1}\right).
\end{equation*}
Furthrmore, $\hat{F}_{m,\chi}(s)$ extends to a meromorphic function in $s$ which, indeed, is
equal to a rational function.
\end{corollary}


\section{Concluding remarks}

\subsection{Cotangent and tangent sums}

In view of the standard identity $\csc^2 x=1+\cot^2 x$, we can also deduce results complementary to
\cite[Theorem 2.2]{He20}, where evaluations of cotangent sums twisted by multiplicative character were obtained.

Specifically, it is clear that computing even powers of cotangent sums reduces to
computing even powers of cosecant sums of the same argument and with the same twist by an additive character. In other words,
an application of the recurrence relation in Corollary \ref{cor: recurrence even powers}
then allows one to compute that
$$
\frac{1}{m}\sum_{j=0}^{m-1} \cot^{2n}\left(\frac{j+\beta}{m}\pi\right) e^{\frac{2\pi i r}{m} j}= \sum_{k=0}^{n}\binom{n}{k}(-1)^{n-k}C_{m,r}(\beta,k),
$$
where we define $C_{m,r}(\beta,0)$ to be equal to $1$ for all values of $m,\,r,\,\beta$.  We assume that $m$ and $r$ are chosen as above and that
$\beta$ is such that $\beta\notin\mathbb{Z}$.
Similar reasoning applies to the computation of cotangent sums without the shift $\beta$ and to the computation of even powers of tangents, which reduces to an application of the binomial theorem to and secant sums.

These results can be compared to those of \cite{EL21} where the authors compute, by using a different method, sums of any powers of cotangent and tangent functions at arguments of the form $\frac{j+\beta}{m}\pi$. Their result is more general in the sense that they treat both even and odd powers. On the other hand, we look only at even powers, but employ a character twist.  As shown above, the use of the character twist is necessary in other situations,
such as  when one wants to apply the Gauss formula and pass to multiplicative character twists; see Section \ref{sec: mult twist} above.

\subsection{Differentiating or integrating with respect to $\beta$}\label{sec: dif int resp to beta}

A further possibility that presents itself is to differentiate or integrate the formulas
above with respect to $\beta$.  Let us illustrate an approach.

Let $\chi$ be a primitive, odd Dirichlet character, from which we seek to study the function
$\tilde{L}_{X_m}(s,\chi)$ defined by \eqref{eq: def of spectral L odd}.  To do so, let us
start with the shifted $L$-function which we define for $\beta\notin\mathbb{Z}$ by
$$
L_{X_m}(s,\chi,\beta)=\sum_{j=0}^{m-1}\chi(j)\csc^{2s}\left(\frac{(j+\beta)\pi }{m}\right)= \sum_{j=1}^{m-1}\chi(j)\csc^{2s}\left(\frac{(j+\beta)\pi }{m}\right).
$$
By proceeding analogously as in the proof of Corollary \ref{cor: twist by even char}, it is
immediate that the generating function
$$
F_{m,\chi}(t,\beta)= \sum_{n=0}^{\infty}(-1)^n 2^{-(n+1)} \overline{L_{X_m}(n+1,\chi, \beta)}t^n
$$
for the special values of $L_{X_m}(s,\chi,\beta)$ at positive integers $s=n$ is given by
\begin{equation*}
F_{m,\chi}(t,\beta) := \frac{m}{\tau(\chi)} \sum_{r=0}^{m-1}\chi(r)F_{m,r}(t,\beta)
\end{equation*}
where $F_{m,r}(t,\beta)$ is defined by \eqref{eq: defn of F}
and $t$ is any complex value where $\vert t \vert$ is sufficiently small and $\mathrm{Re}(t)\geq 0$.
On the other hand, for any non-zero $s$ one has that
$$
\frac{\partial}{\partial\beta}\left.L_{X_m}(s,\chi,\beta) \right|_{\beta=0} = \frac{-2s \pi}{m}\sum_{j=1}^{m-1}\chi(j)\csc^{2s}\left(\frac{j\pi}{m}\right)\cot\left(\frac{j\pi}{m}\right)
= \frac{-2s \pi}{m} \tilde{L}_{X_m}(s,\chi).
$$
It is evident that the generating function for the values of $\tilde{L}_{X_m}(s,\chi)$ at
positive integers can be expressed in terms of derivatives of  $F_{m,\chi}(t,\beta)$ with respect to
$\beta$ evaluated as $\beta \to 0$.

For the sake of limiting the length of our paper we have not pursued the computations here. In fact, apart
from deriving certain new formulas of special interest, the main goal of our paper is to provide
a general method and framework by which one can evaluate a wealth of finite trigonometric sum
rather than to catalogue all such formulas
that can be established in this way.



\subsection{Chebyshev polynomials}\label{appendix: Cheb pol}

For the convenience of the reader,
we recall some notation and  properties of Chebyshev polynomials which are needed above.

Chebyshev polynomials of the first kind are defined for $x\in[-1,1]$ and positive integers $n$ by the relation $T_n(x) = \cos(n \cos^{-1} x)$. For $|x|\geq 1$ the Chebyshev polynomials of the first kind are defined for positive integers $n$ by the relation
$$
T_n(x)=\frac{1}{2}\left((x-\sqrt{x^2-1})^n + (x+\sqrt{x^2-1})^n\right).
$$
By the principle of analytic continuation, we may assume that $T_n$ is defined for positive integers $n$ and complex numbers $z$
with $\mathrm{Re}(z)\geq 1$ by
$$
T_n(z)=\frac{1}{2}\left((z-\sqrt{z^2-1})^n + (z+\sqrt{z^2-1})^n\right)=\cosh(n\cosh^{-1}(z)),
$$
where we use the principal branch of the square root.

Chebyshev polynomials of the second kind are defined for $x\in[-1,1]$ and positive integers $n$ by the relation
$U_n(x) \sqrt{1-x^2}= \sin((n+1)\cos^{-1}x)$. By extending the definition to $x$ with $|x|\geq 1$ and then using
the principle of analytic continuation, it is easy to see that $U_n$ is defined for positive integers $n$ and
complex numbers $z$ with $\mathrm{Re}(z)\geq 1$ by
$$
U_n(z)=\frac{1}{2\sqrt{z^2-1}}\left((+-\sqrt{z^2-1})^{n+1} - (z-\sqrt{z^2-1})^{n+1}\right)=\frac{\sinh((n+1)\cosh^{-1}(z))}{\sqrt{z^2-1}},
$$
where we use the principal branch of the square root.

For all positive integers $n$, the functions $T_n(z)$ and $U_n(z)$ are holomorphic at $z=1$. Moreover, let us set
$$
T_n(z)=\sum_{k=0}^{\infty}a_n(k)(z-1)^k \quad \text{and} \quad U_n(z)=\sum_{k=0}^{\infty}b_n(k)(z-1)^k.
$$
Formula 8.949.2 from \cite{GR07} expresses the derivatives of $T_n(z)$ in terms of the
Gegenbauer polynomials, and formula 8.937.4 from \cite{GR07} evaluates the Gegenbauer
polynomials at $z=1$.  When combining these reesult,  we arrive at the equality $a_n(0)=1$ and, moreover, that
\begin{equation}\label{eq: coeff of T_n}
a_n(k)=\frac{1}{k!} \cdot \prod_{j=0}^{k-1}\frac{n^2-j^2}{(2j+1)}
\,\,\,\,\,
\text{\rm for $0 \leq k \leq n$.}
\end{equation}
One can proceed in a similar way, this time by applying 8.949.5 from \cite{GR07} which expresses the derivatives of $U_n(z)$ in terms of the
Gegenbauer polynomials.  In doing so, we arrive at the equality that
\begin{equation}\label{eq: coeff of U_n}
b_n(k)= \frac{1}{(n+1)k!} \cdot \prod_{j=0}^{k}\frac{(n+1)^2-j^2}{(2j+1)}
\,\,\,\,\,
\text{\rm for $0 \leq k \leq n$.}
\end{equation}

Let
$$
T_n(z) = \sum_{j=0}^n t_n(j) z^j
\,\,\,\,\,
\text{\rm and}
\,\,\,\,\,
U_n(z) = \sum_{j=0}^n u_n(j) z^j.
$$
Formula 8.994 from \cite{GR07} implies that  $t_n(0)=u_n(0)=0$ for odd $n$ and $t_{2k}(0)=u_{2k}(0)=(-1)^k$ for even integers $n=2k$.
The formulas 9.392.2 and 8.392.3 from \cite{GR07} express the Gegenbauer polynomials in terms of the hypergeometric function.
When combining with formulas 8.949.2 and 8.949.2 from \cite{GR07}, we get expressions for the coefficients
$t_{n}(j)$ and $u_{n}(j)$ as follows.
First, assume that $n=2k+1>0$ is odd.  Then $t_n(j)=u_n(j)=0$ for all even $0\leq j<n$, and for odd values of $1\leq j \leq n$ one has
\begin{equation}\label{eq: coeff of T_n at zero}
t_n(j)=(-1)^{\frac{n-j}{2}}\frac{n}{n+j}\binom{\frac{n+j}{2}}{\frac{n-j}{2}}\cdot 2^j \quad \text{and} \quad u_n(j)=(-1)^{\frac{n-j}{2}}\binom{\frac{n+j}{2}}{\frac{n-j}{2}}\cdot 2^j.
\end{equation}
When $n=2k$ is even, then $t_n(j)=u_n(j)=0$ for all odd values of $1\leq j < n$, while for even values of $0\leq j\leq n$ the
coefficients $t_n(j)$ and $u_n(j)$ are given by \eqref{eq: coeff of T_n at zero}.

\vspace{5mm}\noindent
Jay Jorgenson \\
Department of Mathematics \\
The City College of New York \\
Convent Avenue at 138th Street \\
New York, NY 10031
U.S.A. \\
e-mail: jjorgenson@mindspring.com

\vspace{5mm}\noindent
Anders Karlsson\\
Section de math\'ematiques \\
Universit\'e de Gen\`eve \\
Case postale 64,
1211 Gen\`eve, Switzerland; \\
Mathematics department \\
Uppsala University \\
Box 256, 751 05 Uppsala, Sweden \\
e-mail: anders.karlsson@unige.ch

\vspace{5mm}\noindent
Lejla Smajlovi\'c \\
Department of Mathematics \\
University of Sarajevo\\
Zmaja od Bosne 35, 71 000 Sarajevo\\
Bosnia and Herzegovina\\
e-mail: lejlas@pmf.unsa.ba

\end{document}